\documentclass[10pt,reqno]{amsart}

\usepackage{amsmath,amscd,amssymb }

\usepackage[backend=bibtex,style=numeric]{biblatex}
\usepackage{enumerate}
\bibliography{mabuchi.bib}

\newtheorem{proposition}{Proposition}[section]
\newtheorem{lemma}[proposition]{Lemma}
\newtheorem{theorem}[proposition]{Theorem}
\newtheorem{corollary}[proposition]{Corollary}
\newtheorem{conjecture}{Conjecture}[section]

\theoremstyle{definition}
\newtheorem{remark}[proposition]{Remark}
\newtheorem{definition}[proposition]{Definition}

\newcommand{\R}{\mathbb{R}}

\newcommand{\Q}{\mathbb{Q}}

\newcommand{\pr}{\mathbb{P}}

\newcommand{\scH}{\mathcal{H}}
\newcommand{\scM}{\mathcal{M}}

\newcommand{\scO}{\mathcal{O}}

\DeclareMathOperator{\Aut}{Aut}

\DeclareMathOperator{\Ric}{Ric}

\title[Alpha invariants and the Mabuchi functional]{Alpha invariants and coercivity of the Mabuchi functional on Fano manifolds }

\author{Ruadha\'i Dervan}

\begin{document}

\maketitle

\begin{abstract}
We give a criterion for the coercivity of the Mabuchi functional for general K\"ahler classes on Fano manifolds in terms of Tian's alpha invariant. This generalises a result of Tian in the anti-canonical case implying the existence of a K\"ahler-Einstein metric. We also prove the alpha invariant is a continuous function on the K\"ahler cone. As an application, we provide new K\"ahler classes on a general degree one del Pezzo surface for which the Mabuchi functional is coercive.
\end{abstract}


\section{Introduction}

A central problem in K\"ahler geometry is to understand the existence of constant scalar curvature K\a"ahler (cscK) metrics in a fixed K\"ahler class on a K\"ahler manifold. The constant scalar curvature equation is a fully nonlinear fourth order PDE, which in general has proven very difficult to understand. An important idea due to Mabuchi is that one can attempt to understand the existence of cscK metrics through properties of an associated functional, now called the Mabuchi functional \cite{TM}. This is a functional on the space of K\"ahler metrics in a fixed K\"ahler class, and has critical points the metrics of constant scalar curvature, when they exist. Defining a Riemannian metric on the space of K\"ahler metrics, one sees that the Mabuchi functional is convex along geodesics.  Motivated by this picture, Mabuchi conjectured that the existence of a cscK metric should be equivalent to the Mabuchi functional being bounded below \cite{TM}. This conjecture was refined by Tian to require coercivity, which is slightly stronger than boundedness.

\begin{conjecture}[Tian]\cite[Conjecture 7.12]{GT3}\label{mabuchiconjecture} A K\"ahler manifold $(X,\omega)$ admits a cscK metric if and only if the Mabuchi functional is coercive. \end{conjecture}

It is now known that the existence of a cscK metric implies boundedness of the Mabuchi functional \cite{BB,CT}. Conjecture \ref{mabuchiconjecture} was proven by Tian when $\omega\in c_1(X)$ \cite[Theorem 7.13]{GT3}, so that $X$ is a Fano manifold and the metric is K\"ahler-Einstein, however less is known in the general case. 

When $\omega\in c_1(X)$, Tian introduced a \emph{sufficient} condition for the existence of a K\"ahler-Einstein metric using the alpha invariant. Let $\omega$ be a K\"ahler form and denote the space of K\"ahler potentials by $$\scH(\omega) = \{\phi\in C^{\infty}(X,\R): \omega+i\partial\bar{\partial}\phi > 0\}.$$ \begin{definition}The alpha invariant $\alpha(X,[\omega])$ of $(X,[\omega])$ is defined as $$\alpha(X,[\omega])) = \sup\left\{ \beta: \int_X e^{-\beta(\phi - \sup_X \phi)}\omega^n < c \right\},$$ for some $c$ independent of $\phi\in \scH(\omega)$. When $\omega\in c_1(L)$ for some ample line bundle $L$, we denote by $\alpha(X,L)$ the corresponding alpha invariant. \end{definition} 

\begin{theorem}\label{tianalpha}\cite{GT1} Let $X$ be an $n$-dimensional Fano manifold with $\alpha(X,c_1(X))>\frac{n}{n+1}$. Then $X$ admits a K\"ahler-Einstein metric. \end{theorem}

This criterion is one of very few methods of constructing K\"ahler-Einstein metrics, and was used in a fundamental way in Tian-Yau's classification of K\"ahler-Einstein metrics on del Pezzo surfaces \cite{GT2,TY}. Tian's original proof of Theorem \ref{tianalpha} used the continuity method, however the above Theorem can also be proven by showing that the alpha invariant condition implies that the Mabuchi functional is coercive, which in turn implies the existence of a K\"ahler-Einstein metric \cite[Theorem 7.13]{GT3}. 

We generalise Tian's criterion for coercivity of the Mabuchi functional to general K\"ahler classes on some manifolds. In order to ease notation, define the \emph{slope} of a K\"ahler manifold $(X,\omega)$ to be $$\mu(X,[\omega]) = \frac{\int_X c_1(X).[\omega]^{n-1}}{\int_X [\omega]^{n}}.$$ Note that the slope is a topological quantity which, by scaling $[\omega]$, can be assumed equal to one. For $G$ a compact subgroup of $\Aut(X,[\omega])$, we similarly define $\alpha_G(X,[\omega])$ by considering only $G$-invariant K\"ahler potentials. Similarly we define the appropriate notion of coercivity of the Mabuchi functional, again considering only $G$-invariant K\"ahler potentials. Our main result is as follows.

 \begin{theorem}\label{intromaintheorem} Let $(X,\omega)$ be a K\"ahler manifold of dimension $n$. Suppose that \begin{itemize} \item[(i)] $\alpha_G(X,[\omega])>\mu(X,[\omega]) \frac{n}{n+1}$, \item[(ii)] $c_1(X)\geq \frac{n}{n+1}\mu(X,[\omega]) [\omega]$, in the sense that the difference is nef. \end{itemize} Then the Mabuchi functional is coercive on the space of $G$-invariant K\"ahler potentials for $(X,[\omega])$. \end{theorem}

The second condition gives an explicit neighbourhood of the anti-canonical bundle of a Fano manifold for which Theorem \ref{intromaintheorem} can be applied. Theorem \ref{intromaintheorem} reduces to Tian's criterion in the case $[\omega]=c_1(X)$, as in this case the slope condition is automatic. 

Note that the criteria of Theorem \ref{intromaintheorem} do not imply that $X$ is projective, for example they hold for non-projective Calabi-Yau manifolds. Moreover, even when $X$ is projective, they do not imply that $\omega\in c_1(L)$ for some ample line bundle $L$. On the other hand, when $X$ is projective and the slope is non-negative, condition $(ii)$ and the Hodge Index Theorem imply that $X$ must either be numerically Calabi-Yau or Fano \cite[Remark 3.5]{RD1}. 

When $c_1(X)<0$, similar results to Theorem \ref{intromaintheorem} have been intensively studied using an auxiliary equation called the J-flow \cite{XC, SW,BW}. In particular, a result due to Weinkove \cite{BW} implies Theorem \ref{intromaintheorem} in the case of K\"ahler manifolds with negative first Chern class. In contrast with the techniques involved in studying the J-flow, our proof of Theorem \ref{intromaintheorem} is by a direct analysis of the Mabuchi functional. As such, we assume throughout that $\mu(X,[\omega])>0$ so that $X$ is Fano, though we emphasise that the techniques involved in proving our result apply also to the case $c_1(X)\leq 0$.

For a first application of Theorem \ref{intromaintheorem}, we prove in Section \ref{continuitysection} a continuity result for the alpha invariant, extending a result of the author in the case of the ample cone of a projective variety \cite[Proposition 4.2]{RD1}. The proof in the algebraic case uses a characterisation of the alpha invariant in terms of log canonical thresholds of divisors due to Demailly \cite[Theorem A.3]{Demailly}. The algebraic proof extends in a straightforward manner to the case of $\R$-line bundles, assuming Demailly's characterisation holds in this setting. We use instead a characterisation in terms of the complex singularity exponent of singular metrics, which avoids this issue and proves continuity in the (possibly) larger K\"ahler cone.

\begin{theorem}\label{continuity} The alpha invariant is a continuous function on the K\"ahler cone.\end{theorem}

It follows that the set of K\"ahler classes satisfying the hypotheses of Theorem \ref{intromaintheorem} is open, provided strict inequality holds in the condition (ii). One should compare this to a result of LeBrun-Simanca \cite{LS}, which states that the existence of a cscK metric is an open condition in the K\"ahler cone of a K\"ahler manifold, provided the derivative of the Futaki invariant vanishes. 

One can also use Theorem \ref{intromaintheorem} to give \emph{explicit} K\"ahler classes with coercive Mabuchi functional. Let $X$ be a general del Pezzo surface of degree one, so that $X$ is isomorphic to the blow-up of $\pr^2$ at eight points in general position. Let $$L_{\lambda} = \pi^*(\scO_{\pr^2}(1))\otimes \scO(-E_1)\otimes\hdots\otimes\scO(-E_7)\otimes\scO(-\lambda E_8),$$ where $\pi: X \to \pr^2$ is the natural blow-up morphism and $E_i$ are the exceptional divisors, and we allow $\lambda$ to be a \emph{real} number. The criteria of Theorem \ref{intromaintheorem} were calculated in \cite[Theorem 1.2]{RD1} for $\lambda\in\Q$, and extend to $\lambda\in\R$ by Theorem \ref{continuity}, giving the following. \begin{corollary}\label{application} The Mabuchi functional on $(X,L_{\lambda})$ is coercive for $$\frac{19}{25}\approx \frac{1}{9}(10-\sqrt{10}) < \lambda < \sqrt{10}-2\approx \frac{29}{25}.$$ \end{corollary} 

Further examples could be obtained from \cite{TD, LSY}, where alpha invariants of general ample line bundles on toric varieties are calculated. 

When $\omega\in c_1(L)$ for some ample line bundle $L$, the Yau-Tian-Donaldson conjecture states that the existence of a cscK metric should be equivalent to the algebro-geometric notion of K-stability \cite{SD}. Theorem \ref{intromaintheorem} is motivated by the following result due to the author, which proves K-stability under the same hypotheses. 

\begin{theorem}\label{kstab}\cite[Theorem 1.1]{RD1}\cite[Theorem 1.9]{RD2} Let  $X$ be a smooth variety together with an ample line bundle $L$. Denote by $K_X$ the canonical bundle of $X$. Suppose that 
\begin{enumerate}[(i)] 
\item $\alpha(X,L)>\frac{n}{n+1}\mu(X,L)$,
\item $-K_X \geq \frac{n}{n+1}\mu(X,L) L$, in the sense that the difference is nef.
\end{enumerate}
Then $(X,L)$ is uniformly K-stable with respect to the minimum norm.\end{theorem}

The uniformity in the above result should be thought of as analogous to coercivity of the Mabuchi functional \cite[Conjecture 1.1]{RD2}, it in particular implies K-stability.

Theorem \ref{intromaintheorem} should be compared to recent work of Li-Shi-Yao, who proved a similar criterion for the coercivity of the Mabuchi functional, using a link with the convergence of the J-flow \cite[Theorem 1.1]{LSY}.

\begin{theorem}\label{lsytheorem}Let $(X,\omega)$ be an n-dimensional compact K\"ahler manifolds. Suppose the K\"ahler class $[\omega]$ satisfies the following conditions for some $\epsilon\geq 0$. 
\begin{enumerate}[(i)] 
\item $\alpha(X,[\omega])>\frac{n}{n+1}\epsilon,$
\item $\epsilon [\omega]>c_1(X),$
\item $$\left(-n\frac{\int_X c_1(X).[\omega]^{n-1}}{\int_X [\omega]^n} +\epsilon \right)[\omega] + (n-1)c_1(X) > 0.$$
\end{enumerate}
Then the Mabuchi functional is coercive on the class $[\omega]$. \end{theorem}

Li-Shi-Yao \cite{LSY} apply their result to a general del Pezzo surface of degree one, showing that their criterion implies the Mabuchi functional is coercive for $$\frac{4}{5}<\lambda < \frac{10}{9},$$ with notation as in Corollary \ref{application}. Note that this interval is strictly contained in the interval obtained in Corollary \ref{application}. Since both criteria are independent of scaling $[\omega]$ linearly, it is natural to expect in general that Theorem \ref{intromaintheorem} is stronger than Theorem \ref{lsytheorem} on Fano manifolds. While their proof relies on properties of an analysis of the J-flow, our proof of Theorem \ref{intromaintheorem} is direct.

Finally, we remark that Theorem \ref{intromaintheorem} can be extended to the setting of the twisted Mabuchi functional and the log Mabuchi functional, which are related to the existence of twisted cscK metrics and cscK metrics with cone singularities along a divisor respectively. The proof is entirely similar, as such we leave the details to the interested reader. The analogous results in terms of twisted K-stability and log K-stability are \cite[Theorem 1.9]{RD2} and \cite[Theorem 1.3]{RD1}.

\

\noindent {\bf Notation and conventions:} We take $X$ throughout to be a compact K\"ahler manifold, and omit certain factors of $2\pi$ for notational convenience.

\

\noindent {\bf Acknowledgements:} I would like to thank my supervisor, Julius Ross, as well as David Witt Nystr\"om and Yoshinori Hashimoto for helpful comments. Supported by a studentship associated to an EPSRC Career Acceleration Fellowship (EP/J002062/1) and a Fondation Wiener-Anspach scholarship.

\section{Coercivity of the Mabuchi functional}

Let $(X,\omega)$ be an $n$-dimensional K\"ahler manifold. Denote by $$\scH(\omega) = \{\phi: \omega_{\phi}=\omega+i\partial\bar{\partial}\phi > 0\}$$ the space of K\"ahler potentials in a fixed K\"ahler class, where we have fixed some basepoint $\omega$. The space of K\"ahler potentials can therefore be identified with a convex subset of $C^{\infty}(X,\R)$. 

\begin{definition}\cite{TM} Let $\phi$ be a K\"ahler potential with K\"ahler form $\omega_{\phi}$, and let $\omega_t$ be any path connecting $\omega=\omega_0$ to $\omega_{\phi}=\omega_1$. Let $\phi_t$ be the corresponding K\"ahler potentials. The \emph{Mabuchi functional} is the functional on $\scH(\omega)$ defined by $$\scM_{\omega}(\phi)=-\int_0^1 \int_X \dot{\phi}_t(S(\omega_t) - n\mu(X,[\omega]))\omega_t^n\wedge dt,$$ where $S(\omega_t)$ is the scalar curvature. Explicitly, the Mabuchi functional is the sum of an \emph{entropy term} and an \emph{energy term} $$\scM_{\omega}(\phi) = H_{\omega}(\phi)+E_{\omega}(\phi),$$ where respectively \begin{align*} H_{\omega}(\phi) &=\int_X \log(\frac{\omega_{\phi}^n}{\omega})\omega_{\phi}^n, \\ E_{\omega}(\phi) &= \mu(X,[\omega])\frac{n}{n+1} \sum_{i=0}^n\int_X \phi \omega^i \wedge \omega^{n-i}_{\phi}  - \sum_{i=0}^{n-1}  \int_X \phi \Ric \omega \wedge \omega^i \wedge \omega_{\phi}^{n-1-i}.\end{align*} In particular the Mabuchi functional is independent of the path chosen connecting $\omega$ to $\omega_{\phi}$.\end{definition}

We will also require Aubin's $I$-functional.

\begin{definition} \emph{Aubin's} $I$ \emph{functional} is defined to be $$I_{\omega}(\phi) =\int_X \phi (\omega^n - \omega_{\phi}^n).$$ \end{definition}

\begin{lemma}\label{Ilemma}\cite[p. 46]{GT3} Aubin's functional is positive, i.e. $I_{\omega}(\phi)\geq 0$.\end{lemma}

\begin{proof} Indeed, we have \begin{align*}I_{\omega}(\phi)&=\int_X \phi (\omega^n - \omega_{\phi}^n), \\ &=\int_X \phi (\omega- \omega_{\phi})(\omega_{\phi}^{n-1} +\omega_{\phi}^{n-2}\wedge\omega+\hdots+ \omega^n), \\ &=\int_X \phi (-i\partial\bar{\partial}\phi)(\omega_{\phi}^{n-1} +\omega_{\phi}^{n-2}\wedge\omega+\hdots+ \omega^n), \\ &=\int_X i\partial\phi \wedge \bar{\partial}\phi\wedge(\omega_{\phi}^{n-1} +\omega_{\phi}^{n-2}\wedge\omega+\hdots+ \omega^n), \\ &\geq 0. \end{align*} \end{proof}

With these definitions in place we can define the coercivity of the Mabuchi functional.

\begin{definition}\cite[Section 7.2]{GT3} We say the Mabuchi functional is \emph{coercive} if $$\scM_{\omega}(\phi) \geq a I_{\omega}(\phi) + b,$$ where $a,b\in\R$ are constants independent of $\phi$ with $a>0$. Similarly we say that the Mabuchi functional is \emph{bounded} if for some constant $c\in\R$ we have $$\scM_{\omega}(\phi) \geq c.$$ By Lemma \ref{Ilemma} coercivity of the Mabuchi functional implies boundedness of the Mabuchi functional. \end{definition}

\begin{remark} This notion of coercivity is sometimes also called properness of the Mabuchi functional in the literature.\end{remark}

The following result is well known.

\begin{lemma}\label{normalisation} To prove boundedness or coercivity of the Mabuchi functional, one can restrict to K\"ahler potentials $\phi$ with $\sup_X \phi = 0$. \end{lemma}

\begin{proof} Note that if $\phi$ is a K\"ahler potential, then so is $\phi+c$ for all $c \in \R$. 

Through its variational definition, one sees that the Mabuchi functional is independent of $\phi\to \phi + c$ for any $c\in\R$, that is, $\scM(\phi) = \scM(\phi+c)$. Under the transformation $\phi \to \phi - \sup_X \phi$, this shows that to prove boundedness of the Mabuchi functional, one can restrict to K\"ahler potentials with $\sup_X \phi = 0$. A direct computation show that $I_{\omega}(\phi)=I_{\omega}(\phi+c)$. 

Since both the Mabuchi functional and Aubin's $I$-functional are invariant under the addition of constants, the transformation $$\phi \to \phi - \sup_X \phi$$ provides the required coercivity result.  \end{proof}

Roughly speaking, to control the Mabuchi functional, we use the positivity of the entropy term to control the energy term. 

\begin{lemma}\cite[Theorem 7.13]{GT3}\label{alphalemma} Suppose $\sup_X \phi = 0$. Then for all $\beta < \alpha(X,L)$, the entropy term satisfies $$H_{\omega}(\phi) \geq -\beta \int_X \phi \omega_{\phi}^n + C,$$ for some $C$ independent of $\phi$. \end{lemma}

\begin{proof} We recall Tian's proof. Since we have assumed $\sup_X \phi = 0$, the alpha invariant satisfies $$\alpha(X,[\omega])) = \sup\left\{  \beta: \int_X e^{-\beta\phi}\omega^n < c \right\}.$$ Therefore for any $\beta < \alpha(X,L)$ the integral $\int_X e^{-\beta\phi}\omega^n$ is bounded by some constant $c$ independent of $\phi$. We have $$ \int_X e^{-\beta\phi}\omega^n = \int_X  e^{-\log \frac{\omega_{\phi}^n}{\omega^n}-\beta\phi}\omega_{\phi}^n\leq c,$$ which by Jensen's inequality implies $$\int_X \left(-\log \frac{\omega_{\phi}^n}{\omega^n} - \beta\phi\right)\omega_{\phi}^n \leq \log c.$$ Finally this gives $$\int_X \log \frac{\omega_{\phi}^n}{\omega^n} \omega_{\phi}^n \geq -\beta\int_X \phi \omega_{\phi}^n - \log c,$$ as required with $C = -\log c$.

\end{proof}

The following Lemma will be key to proving coercivity of the Mabuchi functional, rather than just boundedness.

\begin{lemma}\label{toprovecoercivity} Suppose $\sup_X \phi = 0$. Then  $$-\int_X \phi \omega_{\phi}^n\geq I_{\omega}(\phi).$$ \end{lemma}

\begin{proof} Since $I_{\omega}(\phi)=\int_X \phi (\omega^n - \omega_{\phi}^n)$, we wish to show that $$\int_X \phi \omega^n \leq 0.$$ To see this, note that since $\omega$ is a positive $(1,1)$-form we have \begin{align*}\int_X \phi \omega^n &\leq \int_X\sup_X \phi \omega^n, \\ &= \sup_X \phi \int_X \omega^n, \\ &= 0.\end{align*}\end{proof}

We will also need the following technical Lemma, which will again allow us to utilise the positivity of the $-\int_X \phi \omega_{\phi}^n$ term.

\begin{lemma}\label{LnE} Suppose $\sup_X \phi = 0$. Then $$-n\int_X \phi \omega_{\phi}^n \geq -\sum_{i=1}^{n}  \int_X \phi  \omega^i \wedge \omega_{\phi}^{n-i}.$$ \end{lemma}

\begin{proof}

We write \begin{align*}-n\int_X \phi \omega_{\phi}^n + \sum_{i=1}^{n}  \int_X \phi  \omega^i \wedge \omega_{\phi}^{n-i} &=  \sum_{i=1}^{n} \int_X \phi (-\omega_{\phi}^n + \omega^i \wedge \omega_{\phi}^{n-i}), \\ &=  \sum_{i=1}^{n} \int_X \phi \omega_{\phi}^{n-i}\wedge( \omega^i  -  \omega_{\phi}^{i}), \end{align*} and show that each summand is positive separately. Using the same technique as Lemma \ref{Ilemma}, we have \begin{align*}\int_X \phi \omega_{\phi}^{n-i}\wedge\left( \omega^i  -  \omega_{\phi}^{i}\right) &= \int_X \phi \omega_{\phi}^{n-i}\wedge( \omega -  \omega_{\phi})\wedge\left( \sum_{j=1}^{i}\omega^{i-j} \wedge \omega_{\phi}^{j-1}\right), \\ &=\int_X \phi( -i\partial\bar{\partial}\phi) \wedge \omega_{\phi}^{n-i}\wedge\left( \sum_{j=1}^{i}\omega^{i-j} \wedge \omega_{\phi}^{j-1}\right), \\ &=\int_X i\partial\phi \wedge \bar{\partial}\phi\wedge\omega_{\phi}^{n-i}\wedge\left( \sum_{j=1}^{i}\omega^{i-j} \wedge \omega_{\phi}^{j-1}\right), \\ &\geq 0. \end{align*} 

 \end{proof}

We now proceed to the proof of Theorem \ref{intromaintheorem}.

\begin{proof}[Proof of Theorem 1.3] 

We assume $G$ consists of just the identity automorphism, for notational convenience. The general case is the same.

By Lemma \ref{normalisation}, it suffices to show that the Mabuchi functional is coercive with respect to all $\phi$ satisfying $\sup_X \phi = 0.$ 

The Mabuchi functional is given as \begin{align*}\scM(\phi) =\int_X \log\frac{\omega_{\phi}^n}{\omega^n}\omega_{\phi}^n& + \frac{n\mu(X,[\omega])}{n+1}\int_X \phi \left(\sum^n_{i=0} \omega^i \wedge \omega^{n-i}_{\phi}\right) - \\ & -\int_X \phi \Ric \omega \wedge \left(\sum^{n-1}_{i=0} \omega^i \wedge \omega^{n-1-i}_{\phi}\right).\end{align*}

Let $\alpha(X,L) = \frac{n}{n+1}\mu(X,[\omega]) + (n+2)\epsilon$, where by assumption $\epsilon>0$. Since $\sup_X \phi = 0$, Lemma \ref{alphalemma} implies $$H(\phi) \geq -\left(\frac{n}{n+1}\mu(X,[\omega]) + (n+1)\epsilon\right)\int_X \phi \omega_{\phi}^n + c.$$ Here $c$ is a constant independent of $\phi$. Using this, we see \begin{align*}&\scM(\phi) \geq -\left(\frac{n}{n+1}\mu(X,[\omega]) + (n+1)\epsilon\right)\int_X \phi \omega_{\phi}^n + \\ & +  \frac{n\mu(X,[\omega])}{n+1}\int_X \phi \left(\sum^n_{i=0} \omega^i \wedge \omega^{n-i}_{\phi}\right) -\int_X \phi \Ric \omega \wedge \left(\sum^{n-1}_{i=0} \omega^i \wedge \omega^{n-1-i}_{\phi}\right) + c. \end{align*} This gives \begin{align*}\scM(\phi) \geq -(n+1)\epsilon\int_X \phi \omega_{\phi}^n& +\frac{n\mu(X,[\omega])}{n+1}\int_X \phi\omega \wedge \left(\sum^{n-1}_{i=0} \omega^i \wedge \omega^{n-1-i}_{\phi}\right)  - \\ &-\int_X \phi \Ric \omega \wedge \left(\sum^{n-1}_{i=0} \omega^i \wedge \omega^{n-1-i}_{\phi}\right)+c.\end{align*} 

From Lemma \ref{LnE}, we see \begin{align*}\scM(\phi) \geq -\epsilon\int_X \phi \omega_{\phi}^n& +\frac{n\mu(X,[\omega])}{n+1}\int_X \phi\omega \wedge \left(\sum^{n-1}_{i=0} \omega^i \wedge \omega^{n-1-i}_{\phi}\right)  - \\ &-\int_X \phi (\Ric \omega + \epsilon \omega) \wedge \left(\sum^{n-1}_{i=0} \omega^i \wedge \omega^{n-1-i}_{\phi}\right) + c.\end{align*} 

Write $$\Ric\omega + \epsilon \omega = \frac{n}{n+1}\mu(X,[\omega])\omega + \eta,$$ so that $\eta-\epsilon \omega \in \left(c_1(X)-\frac{n}{n+1}\mu(X,[\omega]) [\omega]\right)$. By assumption this class is nef, since $\omega$ is positive this implies that there exists a $\xi \in [\eta]$ which is positive. Using the $\partial\bar{\partial}$-lemma write $$\eta = \xi + i\partial\bar{\partial} \psi.$$ We can therefore write the Ricci curvature of $\omega$ as $$\Ric\omega + \epsilon \omega = \frac{n}{n+1}\mu(X,[\omega])\omega + \xi+ i\partial\bar{\partial} \psi.$$ After some cancellation the Mabuchi functional therefore has a lower bound  $$\scM(\phi) \geq-\epsilon\int_X \phi \omega_{\phi}^n -\int_X \phi (\xi + i\partial\bar{\partial} \psi) \wedge \left(\sum^{n-1}_{i=0} \omega^i \wedge \omega^{n-1-i}_{\phi}\right) + c.$$ 

Lemma \ref{toprovecoercivity} gives $-\epsilon\int_X \phi \omega_{\phi}^n \geq \epsilon I_{\omega}(\phi)$. The Mabuchi functional is therefore coercive as long as we can bound the remaining terms by a constant. We bound the two remaining terms individually.

We first consider the $\xi$ term. Since $\xi$ is positive and $\sup_X \phi=0$, we have \begin{align*} \int_X\phi \xi \wedge \left(\sum^{n-1}_{i=0} \omega^i \wedge \omega^{n-1-i}_{\phi}\right) &\leq \int_X \sup_X \phi \xi \wedge \left(\sum^{n-1}_{i=0} \omega^i \wedge \omega^{n-1-i}_{\phi}\right), \\ &=\sup_X \phi\int_X  \xi \wedge \left(\sum^{n-1}_{i=0} \omega^i \wedge \omega^{n-1-i}_{\phi}\right), \\ &= 0.\end{align*} 

Finally, we bound the $\psi$ term. \begin{align*}\int_X \phi i\partial\bar{\partial} \psi\wedge \left(\sum^{n-1}_{i=0} \omega^i \wedge \omega^{n-1-i}_{\phi}\right) &= \int_X \psi i\partial\bar{\partial} \phi\wedge \left(\sum^{n-1}_{i=0} \omega^i \wedge \omega^{n-1-i}_{\phi}\right), \\ &= \int_X \psi (\omega_{\phi}-\omega)\wedge \left(\sum^{n-1}_{i=0} \omega^i \wedge \omega^{n-1-i}_{\phi}\right), \\ &= \int_X \psi (\omega_{\phi}^n-\omega^n), \\ &\leq \int_X (\sup_X \psi) \omega_{\phi}^n-\int_X (\inf_X \psi)\omega^n,\\ &=\sup_X \psi \int_X  \omega_{\phi}^n-\inf_X \psi\int_X\omega^n. \end{align*} Note that the construction of $\psi$ is independent of $\phi$ and that the integrals $\int_X  \omega_{\phi}^n$ and $\int_X\omega^n$ are topological invariants, so the integral is bounded above by a constant independent of $\phi$. This completes the proof that the Mabuchi functional is coercive. 

\end{proof}

Note that in the proof we essentially ignore the term of the form $-\epsilon\int_X \phi \omega_{\phi}^n$. Using Lemma \ref{LnE} to utilise this term, a small variation on the above arguments gives the following generalisation of Theorem \ref{intromaintheorem}.

\begin{theorem}\label{extension} Write $$\alpha(X,[\omega]) = \frac{n}{n+1} \mu(X,[\omega]) + \epsilon([\omega]),$$ and assume $\epsilon([\omega])>0$. If $$c_1(X) + \left(\frac{\epsilon([\omega])}{n} - \frac{n}{n+1}\mu(X,[\omega])\right)[\omega]$$ is a K\"ahler class, then the Mabuchi functional is coercive on the space of K\"ahler potentials for $(X,[\omega])$.\end{theorem}
 
Remark that the $\epsilon([\omega])$ scales such that the criteria of Theorem \ref{extension} are independent of scaling. This small extension is similar to the K-stability situation of Theorem \ref{kstab} \cite[Remark 12]{RD1}.

\section{Continuity of the alpha invariant}\label{continuitysection}

In this section we prove Theorem \ref{continuity}, which states that the alpha invariant is a continuous function on the K\"ahler cone of a K\"ahler manifold. This generalises a result of the author in the case $[\omega]=c_1(L)$ for some ample line bundle $L$, proven using techniques from birational geometry \cite[Proposition 4.2]{RD1}. To prove this result, we use another definition of the alpha invariant involving complex singularity exponents of plurisubharmonic (psh) functions.

\begin{definition}\cite{DemKoll} Fix a K\"ahler manifold $(X,\omega)$. We say that an $L^1$ function $\varphi$ is $\omega$\emph{-psh} if if it is upper semicontinuous and $\omega + i\partial\bar{\partial}\varphi\geq 0$ in the sense of currents. The \emph{complex singularity exponent} of $\varphi$ is defined as $$c(\varphi) = \sup \left\{ \beta>0 \ |\ \int_X e^{-\beta\varphi}\omega^n<\infty \right\}.$$\end{definition}

\begin{proposition}[Demailly]\label{csealpha}\cite[Proposition 8.1]{VT} The alpha invariant satisfies $$\alpha(X,[\omega])=\inf \{c(\varphi) \ |\ \varphi \textrm{ is } \omega\textrm{-psh}\}.$$\end{proposition}

We will require the following. 

\begin{lemma}\label{scaling}\cite[1.4 (4)]{DemKoll} For $\lambda\in\R_{>0}$, the alpha invariant satisfies the scaling property $$\alpha(X,[\lambda\omega])=\lambda^{-1}\alpha(X,[\omega]).$$\end{lemma}

We will also need the following additivity property of the alpha invariant.

\begin{lemma}\label{addingpositive} Let $\xi$ be a K\"ahler class on a K\"ahler manifold $(X,\omega)$. Then $$\alpha(X,[\omega+\xi])\leq \alpha(X,[\omega]).$$\end{lemma}

\begin{proof} This follows by using the definition provided in Proposition \ref{csealpha}. We show that for every $\omega$-psh function $\varphi$, there is an $(\omega+\eta)$-psh function $\psi$ with $c(\psi)\leq c(\varphi)$. In fact we can take $\psi=\varphi$, since $\omega+\xi+i\partial\bar{\partial}\varphi\geq 0$ as $\xi$ is positive and $\varphi$ is $\omega$-psh.

We therefore wish to show that if $$\int_X e^{-\beta\varphi}(\omega+\xi)^n<\infty,$$ then also $$\int_X e^{-\beta\varphi}\omega^n<\infty.$$ We expand the first integral as $$\sum^n_{j=0}{n \choose j} \int_X e^{-\beta\varphi} \omega^{n-i}\wedge\xi^i.$$ As $\omega$ and $\xi$ are both positive, and $e^{-\beta\varphi}$ is a positive function, if this sum converges then certainly its $i=0$ term $\int_X e^{-\beta\varphi}\omega^n$ is finite. \end{proof}

We can now prove the main result of this section.

\begin{theorem} The alpha invariant is a continuous function on the K\"ahler cone.\end{theorem}

\begin{proof} We fix a K\"ahler manifold $(X,\omega)$ and an auxiliary $(1,1)$-form $\eta$, not necessarily positive. Given $\epsilon>0$ sufficiently small, we wish to show there exists a $\delta>0$ such that $$|\alpha(X,[\omega])- \alpha(X,[\omega+\delta\eta])|<\epsilon,$$ where we will take $\delta$ small enough that $\omega+\delta\eta$ remains K\"ahler. 

Suppose there exists a $0<\gamma<1$ such that both $\gamma\omega - \eta$ and $\gamma\omega + \eta$ are K\"ahler. Remark that then $\omega+\eta$ is also K\"ahler. Lemma \ref{scaling} gives $$(1+\gamma)\alpha(X,(1+\gamma)[\omega]) = \alpha(X,[\omega]).$$

Lemma \ref{addingpositive} implies that adding a K\"ahler metric \emph{lowers} the alpha invariant. We apply this by adding the K\"ahler metric $\gamma\omega - \eta$ to $\omega+\eta$. From Lemma \ref{addingpositive} it then follows that $$\alpha(X,(1+\gamma)[\omega])\leq\alpha(X,[\omega+\eta]).$$ So $$\alpha(X,[\omega])=(1+\gamma)\alpha(X,(1+\gamma)[\omega])\leq(1+\gamma)\alpha(X,[\omega+\eta]),$$ hence $$\alpha(X,[\omega]) - \alpha(X,[\omega+\eta]) \leq \gamma\alpha(X,[\omega+\eta]).$$

A similar argument using that $\gamma\omega+\eta$ is K\"ahler gives \begin{equation}\label{temp}\alpha(X,[\omega]) - \alpha(X,[\omega+\eta]) \geq -\gamma\alpha(X,[\omega+\eta]).\end{equation} In particular, provided $\gamma\omega - \eta$ and $\gamma\omega + \eta$ are K\"ahler we have $$|\alpha(X,[\omega]) - \alpha(X,[\omega+\eta])| \leq \gamma\alpha(X,[\omega+\eta]).$$ Equation (\ref{temp}) also gives $$\alpha(X,[\omega+\eta]) \leq \frac{1}{1-\gamma}\alpha(X,[\omega]),$$ which implies under the same hypotheses \begin{equation}\label{temp2}|\alpha(X,[\omega]) - \alpha(X,[\omega+\eta])| \leq \frac{\gamma}{1-\gamma}\alpha(X,[\omega]).\end{equation}

Let $c$ be such that both $\omega+c\eta$ and $\omega-c\eta$ are K\"ahler, which exists as the K\"ahler condition is open. Given $\epsilon>0$, take $\delta = \frac{c\epsilon}{2\alpha(X,[\omega]) + \epsilon}$. Then both $(\delta c^{-1})\omega - \delta\eta$ and $(\delta c^{-1})\omega + \delta\eta$ are K\"ahler. Setting $\gamma = \delta c^{-1}$, we have $0<\gamma<1$ for $\epsilon$ sufficiently small. One computes $\frac{c\delta^{-1}}{1-c\delta^{-1}} = \frac{\epsilon}{2\alpha(X,[\omega])}.$ Therefore equation (\ref{temp2}) implies $$|\alpha(X,[\omega]) - \alpha(X,[\omega+\delta\eta])| \leq \frac{\epsilon}{2}<\epsilon,$$ completing the proof.

\end{proof}

\begin{remark} The proof of the previous result is similar to the proof in the projective case, which uses a characterisation of the alpha invariant in terms of log canonical thresholds of divisors \cite[Proposition 4.2]{RD1}. \end{remark}

\printbibliography

\vspace{4mm}

\noindent Ruadha\'i Dervan, University of Cambridge,  UK and Universit\'e libre de Bruxelles, Belgium. \\ R.Dervan@dpmms.cam.ac.uk

\end{document}